\documentclass[12pt,leqno]{amsart}
\usepackage{amsmath}
\usepackage{epsfig}
\usepackage{graphicx}
\usepackage{subcaption}
\usepackage{yfonts}
\usepackage{caption}
\usepackage{longtable}
\usepackage{dcolumn}
\usepackage{setspace}
\usepackage[most]{tcolorbox}
\usepackage[colorlinks=true, citecolor=blue]{hyperref}
\definecolor{webred}{rgb}{0.75,0,0}
\definecolor{webgreen}{rgb}{0,0.75,0}
\definecolor{refkey}{gray}{0.75}

\numberwithin{equation}{section}

\newtheorem{theo}{Theorem}[section]
\newtheorem{lem}{Lemma}[section]

\newtheorem{Def}[theo]{Definition}
\theoremstyle{remark}
\newtheorem{rem}{Remark}[section]

\newcommand{\m}{m}
\newcommand{\M}{m}

\newcommand{\de}{\delta}

\newcommand{\ep}{\varepsilon}
\def\R{{\mathbb{R}}}

\def\d{\displaystyle}
\def\e{{\varepsilon}}
\def\G{G_1}
\def\GG{\tilde{G}_1}
\def\H{G_2}
\def\HH{\tilde{G}_2}
\def\p{\partial}

\date{}


\tcbset{
    frame code={}
    center title,
    left=0pt,
    right=0pt,
    top=0pt,
    bottom=0pt,
    colback=gray!10,
    colframe=white,
    width=\dimexpr\textwidth\relax,
    enlarge left by=0mm,
    boxsep=5pt,
    arc=0pt,outer arc=0pt,
    }

\begin{document}

\title[Blow-up result for a weakly coupled system of wave  equations]{Blow-up result for a weakly coupled system of 
two Euler-Poisson-Darboux-Tricomi equations with time derivative nonlinearity}
\author[M.F. Ben Hassen, M. Hamouda and M. A. Hamza]{Mohamed Fahmi Ben Hassen$^{1}$, Makram Hamouda$^{1}$ and Mohamed Ali Hamza$^{1}$}
\address{$^{1}$  Department of Basic Sciences, Deanship of Preparatory Year and Supporting Studies, Imam Abdulrahman Bin Faisal University, P. O. Box 1982, Dammam, Saudi Arabia.}

\medskip

\email{mfhassen@iau.edu.sa (M.F. Ben Hassen)}
\email{mmhamouda@iau.edu.sa (M. Hamouda)} 
\email{mahamza@iau.edu.sa (M.A. Hamza)}

\pagestyle{plain}


\maketitle

\begin{abstract}
We study in this article   the blow-up of solutions to a coupled semilinear wave equations which are characterized by  linear damping terms in the \textit{scale-invariant regime}, time-derivative nonlinearities, mass  and Tricomi terms. The latter are specifically of great interest from both physical and mathematical points of view since they allow the  speeds of propagation to be time-dependent ones. However, we assume in this work that both waves are propagating with the same speeds.  Employing this fact together with other hypotheses on the aforementioned parameters (mass and damping coefficients), we obtain a new blow-up region for the system under consideration, and we show a lifespan estimate of the maximal existence time.
\end{abstract}


\section{Introduction}
\par\quad

We aim in this article to study the following system of two semilinear wave equations coupled via the time-derivative nonlinearities, and containing    scale-invariant damping terms,  mass terms and identical time-dependent speeds of propagation, namely
\begin{align}\label{G-sys}
\begin{cases}\d u_{tt}-t^{2 \m}\Delta u +\frac{\mu_1}{t}u_t+\frac{\nu_1^2}{t^2}u=|\partial_t v|^p, &  x\in \mathbb{R}^N, \ t>1, \vspace{.1cm}\\
\d v_{tt}-t^{2 \M}\Delta v +\frac{\mu_2}{t}v_t+\frac{\nu_2^2}{t^2}v=|\partial_t u|^q, &  x\in \R^N, \ t>1,\\
u(x,1)=\e u_1(x), \ \ v(x,1)=\e v_1(x), & x\in \mathbb{R}^N, \\ u_t(x,1)=\e u_2(x), \ \  v_t(x,1)=\varepsilon v_2(x), & x\in \mathbb{R}^N,
\end{cases}
\end{align} 
where $\mu_i,\nu^2_i \ (i=1,2)$ are nonnegative constants representing, respectively, the damping and the mass terms, and $m \ge 0$ stands for the Tricomi term.  The parameter $\e$ is a positive number ensuring the smallness of the initial
data,  and $u_i,v_i\ (i=1,2)$  are positive functions which are compactly supported on the ball $B_{\R^N}(0,R), R>0$.\\
In the subsequent, we assume that the exponent coefficients in the nonlinear terms stated in the right-hand side of \eqref{G-sys}$_{1,2}$ satisfy  $p, q>1$.

Equations or systems  with Tricomi terms like \eqref{G-sys} play a relevant role, for example,  in  quantum field theory when the space-times are considered in a curved geometry (Einstein-de Sitter and anti-de Sitter space-times,
Friedmann-Robertson-Walker space-times, etc.), gas dynamics, elementary particle physics and in cosmology; see e.g. \cite{Yagdjian} and the references therein for more details. The single equation corresponding to \eqref{G-sys} is studied in \cite{Lai-Palm-24, Pal24} where, to the best of our knowledge,  the name of {\it Euler-Poisson-Darboux-Tricomi} equation appeared for the first time.

Let us give an overview on the existing literature about the coupled wave equations in the context of our model \eqref{G-sys} where the mass, the damping and the Tricomi terms are taken in consideration. To illustrate the influence of the  aforementioned terms on the blow-up of \eqref{G-sys}, we  list  here the  known works in this direction. 
First, letting $m=\mu_1 = \mu_2 = \nu_1=\nu_2= 0$ in \eqref{G-sys} yields the following coupled system:\footnote{ Unlike for \eqref{G-sys}, the starting time can be set here at $t=0$ since  the  sytem \eqref{G-sys-0-0} does not contain the damping and the mass terms which my cause a singularity at $0$. The same observation can be applied to \eqref{G-sys-0} below.}
\begin{align}\label{G-sys-0-0}
\begin{cases} u_{tt}-\Delta u =|v_t |^p, &   
 (x,t) \in \R^N\times[0,\infty),   \vspace{.1cm}\\
 v_{tt}-\Delta v =|u_t |^q, &   
 (x,t) \in \R^N\times[0,\infty), \vspace{.1cm}\\
u(x,0)=\e u_1(x), \ \ v(x,0)=\e  v_1(x), & x\in \mathbb{R}^N, \\ u_t(x,0)=\e u_2(x), \ \  v_t(x,0)=\varepsilon v_2(x), & x\in \mathbb{R}^N.
\end{cases}
\end{align}
Note that the case of a single equation corresponding to \eqref{G-sys-0-0} involves the well-known Glassey exponent which is given by
\begin{equation}\label{Glassey}
p_G=p_G(N):=1+\frac{2}{N-1},
\end{equation}
and it constitutes a threshold between the regions of the blow-up ($p \le p_G$) and the existence ($p>p_G$) of a global solution, the literature on this subject is very extensive; see e.g. \cite{Hidano1,Hidano2,John1,Sideris,Tzvetkov,Zhou1}.\\
For the global existence of solutions to  \eqref{G-sys-0-0}, we refer the reader to \cite{Kubo}. However, the blow-up of \eqref{G-sys-0-0} has been the subject of several works; see e.g. \cite{Dao-Reissig,Ikeda-sys,Palmieri1,Palmieri-Takamura-arx}.  More precisely, the critical (in the sense of interface between blow-up and global existence) curve for $p,q$ is given by
\begin{equation}\label{crit-curve-0}
\Upsilon(N,p,q):=\max (\Lambda(N,p,q), \Lambda(N,q,p))=0,
\end{equation}
where 
\begin{equation}\label{Lambda}
\d \Lambda(N,p,q):= \frac{p+1}{pq-1}-\frac{N-1}{2}.
\end{equation}
Under some assumptions, the solution $(u,v)$ of \eqref{G-sys-0-0} blows up in finite time $T(\e)$ for small initial data (of size $\e$),  namely 
\begin{equation}\label{Teps}
T(\e) \le \left\{
\begin{array}{lll}
C \e^{-\Upsilon(N,p,q)}& \text{if}&\Upsilon(N,p,q)>0,\\
\exp(C \e^{-(pq-1)})& \text{if}&\Upsilon(N,p,q)=0, \ p \neq q,\\
\exp(C \e^{-(p-1)})& \text{if}&\Upsilon(N,p,q)=0, \ p = q.
\end{array}
\right.
\end{equation}

Second, let $m= 0$ in \eqref{G-sys}. This case  leads to the damped coupled wave systems with mass terms and constant speed of propagation ($m=0$) which is studied in \cite{Our9}. Note that the results in \cite{Our9} improve the earlier ones in \cite{Palmieri} and generalize the ones in \cite{Our4}. It is worth mentioning that the study of the aforementioned problem,  we introduce the parameter $\de_i$ ($i=1,2$) which is defined by
\begin{equation}\label{delta}
\de_i=(\mu_i-1)^2-4\nu_i^2,
\end{equation}
and we assume that $\de_i >0$. In  fact,  in \cite{Our9} we prove that there is blow-up for the system  \eqref{G-sys} (with $m=0$) for $p,q$  satisfying
\begin{equation}\label{blow-up-reg}
\Omega(N,\mu_1,\mu_2,p,q):=\max (\Lambda(N+\mu_1,p,q), \Lambda(N+\mu_2,q,p)) \ge 0,
\end{equation}
where $\Lambda$ is given by \eqref{Lambda}.\\
Indeed, for small initial data (of size $\e$), the solution $(u,v)$ of \eqref{G-sys} (with $m=0$) blows up in finite time $T(\e)$ that is bounded as follows:
\begin{equation}\label{Teps}
T(\e) \le \left\{
\begin{array}{lll}
C \e^{-\Omega(N,\mu_1,\mu_2,p,q)}& \text{if}&\Omega(N,\mu_1,\mu_2,p,q)>0,\\
\exp(C \e^{-(pq-1)})& \text{if}&\Omega(N,\mu_1,\mu_2,p,q)=0, \\
\exp(C \e^{-\min \left(\frac{pq-1}{p+1},\frac{pq-1}{q+1}\right)})& \text{if}&\Lambda(N+\mu_1,p,q)=\Lambda(N+\mu_2,p,q)=0.
\end{array}
\right.
\end{equation}

On the other hand, we recall here the known results for the case of a single wave equation inherited from (\ref{G-sys}). Indeed,  consider the Cauchy problem (\ref{G-sys}) with $u=v$ and $m,\mu=\mu_i, \nu^2=\nu_i^2>0$, and let us gradually recall the related works in the presence/absence of these parameters. First, for the massless damped wave equation with constant speed of propagation ($m=0$), we refer the reader to \cite{Our, Our2}, see also \cite{Chen,Palmieri-Tu-2019} for different approaches and \cite{Fino} for the context of space-dependent damping term. Furthermore, the mass case with $m=0$ is studied in \cite{Our3} where it is shown that the mass term has no influence in the upper bound of the range of $p$. \\

For $m>0$ and $\mu=\nu=0$,   blow-up results and lifespan estimate of the solution  of the single equation corresponding to \eqref{G-sys} are proven in \cite{Lai2020}, see also \cite{LP-Tricomi}. This yields a new  region characterized by a plausible candidate for the critical exponent, namely  
\begin{equation}\label{pT}
p \le p_T(N,m):=1+\frac{2}{(1+m)(N-1)-m}, \ \text{for} \ N \ge 2.
\end{equation}
In the context of combined nonlinearities, similar results  are obtained in \cite{CLP,Our5}.

Now, in a more general context, namely $m,\mu,\nu>0$, the damped wave equation with Tricomi and mass terms is considered  in \cite{Our8}  which somehow constitutes an improvement of \cite{HHP1,HHP2} for $-1<m<0$ and without the mass term. More precisely, it is proven in \cite{Our8} that the blow-up happens for all $p \le p_T(N+\frac{\mu}{m+1},m)$. Furthemore, regarding the case $-1<m<0$, the reader may also consult the articles \cite{Palmieri, Palmieri2, Tsutaya}. Later, the spreading out to any $m<0$ is studied in \cite{Tsutaya3, Tsutaya2}. \\

Now, letting $\mu_1 = \mu_2 = \nu_1=\nu_2= 0$ in \eqref{G-sys} we find the following coupled system:
\begin{align}\label{G-sys-0}
\begin{cases} u_{tt}-t^{2m}\Delta u =|\partial_t v|^p, &   
 (x,t) \in \R^N\times[0,\infty), \vspace{.1cm}\\
 v_{tt}-t^{2m}\Delta v =|\partial_t u|^q, &   
 (x,t) \in \R^N\times[0,\infty), \vspace{.1cm}\\
u(x,0)=\e u_1(x), \ \ v(x,0)=\e  v_1(x), & x\in \mathbb{R}^N, \\ u_t(x,0)=\e u_2(x), \ \  v_t(x,0)=\varepsilon v_2(x), & x\in \mathbb{R}^N.
\end{cases}
\end{align}
We introduce  the plausible critical (in the sense of interface between blow-up and global existence) curve for $p,q$ is given by
\begin{equation}\label{crit-curve-0-1}
\Upsilon(\tilde{N}_m,p,q)=0,
\end{equation}
where $\Upsilon(\tilde{N}_m,p,q)$ is defined by \eqref{crit-curve-0}-\eqref{Lambda}, and 
\begin{equation}\label{Lambda1}
\d \tilde{N}_m:=N(m+1)-2m.
\end{equation}
Under some assumptions, the solution $(u,v)$ of \eqref{G-sys-0} blows up in finite time $T(\e)$ for small initial data (of size $\e$),  namely 
\begin{equation}\label{Teps}
T(\e) \le \left\{
\begin{array}{lll}
C \e^{-\Upsilon(\tilde{N}_m,p,q)}& \text{if}&\Upsilon(\tilde{N}_m,p,q)>0,\\
\exp(C \e^{-(pq-1)})& \text{if}&\Upsilon(\tilde{N}_m,p,q)=0, \ p \neq q,\\
\exp(C \e^{-(p-1)})& \text{if}&\Upsilon(\tilde{N}_m,p,q)=0, \ p = q.
\end{array}
\right.
\end{equation}
For other types of nonlinearities in systems similar to \eqref{G-sys-0}, we refer the reader to \cite{Ikeda} for the case of combined nonlinearities and \cite{Fan et al} for  power nonlinearities. \\

The main target of this work is to extend a former study carried out in \cite{Our9} where only the damping and mass terms are considered. Here, we intend to study the influence of the Tricomi terms  on the blow-up of the solutions of   the Cauchy problem (\ref{G-sys}) for $m,\mu_i, \nu_i^2>0$. More precisely, in comparison with the case $m=0$ studied in \cite{Our9}, the understanding  of the role of the Tricomi parameter $m$ on the derivation of the blow-up region and the lifespan estimate is the objective of the present study.  \\

The  rest of the paper is organized as follows. We start in Section \ref{sec-main} by defining  the energy solution of (\ref{G-sys}), and  we end the section with the statement of our main result.  Then, we devote Section \ref{aux} to some useful lemmas which will be employed in Section \ref{sec-ut} that is allocated to the proof of the main result. Finally, in Section \ref{open}, we conclude  the article by stating some remarks and open questions.

\section{Main Results}\label{sec-main}
\par

Before the statement of the main result,  we give the definition of  the energy solution associated with (\ref{G-sys}).
\begin{Def}\label{def1}
 We say that $(u,v)$ is an energy  solution of
 (\ref{G-sys}) on $[1,T)$
if
\begin{displaymath}
\left\{
\begin{array}{l}
u,v\in \mathcal{C}([1,T),H^1(\R^N))\cap \mathcal{C}^1([1,T),L^2(\R^N)), \vspace{.1cm}\\
  \ u_t \in L^q_{loc}((1,T)\times \R^N), \ v_t \in L^p_{loc}((1,T)\times \R^N)
 \end{array}
  \right.
\end{displaymath}
satisfies, for all $\Phi, \tilde{\Phi} \in \mathcal{C}_1^{\infty}(\R^N\times[1,T))$ and all $t\in[1,T)$, the following identities:
\begin{equation}
\label{energysol2}
\begin{array}{l}
\d\int_{\R^N}u_t(x,t)\Phi(x,t)dx-\int_{\R^N}u_2(x)\Phi(x,1)dx  -\int_1^t  \int_{\R^N}u_t(x,s)\Phi_t(x,s)dx \,ds\vspace{.2cm}\\
\d+\int_1^t  s^{2\m}\int_{\R^N}\nabla u(x,s)\cdot\nabla\Phi(x,s) dx \,ds  +\int_1^t  \int_{\R^N}\frac{\mu_1}{s}u_t(x,s) \Phi(x,s)dx \,ds \vspace{.2cm}\\
\d+\int_1^t  \int_{\R^N}\frac{\nu^2_1}{s^2}u(x,s) \Phi(x,s)dx \,ds=\int_1^t \int_{\R^N}|v_t(x,s)|^p\Phi(x,s)dx \,ds,
\end{array}
\end{equation}
and
\begin{equation}
\label{energysol3}
\begin{array}{l}
\d\int_{\R^N}v_t(x,t)\tilde{\Phi}(x,t)dx-\int_{\R^N}v_2(x)\tilde{\Phi}(x,1)dx  -\int_1^t  \int_{\R^N}v_t(x,s)\tilde{\Phi}_t(x,s)dx \,ds \vspace{.2cm}\\
\d +\int_1^t s^{2\M} \int_{\R^N}\nabla v(x,s)\cdot\nabla\tilde{\Phi}(x,s) dx \,ds +\int_1^t  \int_{\R^N}\frac{\mu_2}{s}v_t(x,s) \tilde{\Phi}(x,s)dx \,ds \vspace{.2cm}\\
\d  +\int_1^t  \int_{\R^N}\frac{\nu^2_2}{s^2}v(x,s) \tilde{\Phi}(x,s)dx \,ds=\int_1^t \int_{\R^N}|u_t(x,s)|^q\tilde{\Phi}(x,s)dx \,ds,
\end{array}
\end{equation}
 and the conditions $u(x,1)=\varepsilon u_1(x), v(x,1)=\varepsilon v_1(x)$ are fulfilled in $H^1(\mathbb{R}^N)$.\\
A straightforward computation shows that \eqref{energysol2} and \eqref{energysol3} are respectively equivalent to
\begin{equation}
\begin{array}{l}\label{energysol2-2}
\d \int_{\R^N}\big[u_t(x,t)\Phi(x,t)- u(x,t)\Phi_t(x,t)+\frac{\mu_1}{t}u(x,t) \Phi(x,t)\big] dx \vspace{.2cm}\\
\d \int_1^t  \int_{\R^N}u(x,s)\left[\Phi_{tt}(x,s)-s^{2m}\Delta \Phi(x,s) -\frac{\partial}{\partial s}\left(\frac{\mu_1}{s}\Phi(x,s)\right)+\frac{\nu^2_1}{s^2}\Phi(x,s)  \right]dx \,ds\vspace{.2cm}\\
\d =\int_{1}^{t}\int_{\R^N}|v_t(x,s)|^p\Phi(x,s)dx \, ds + \e \int_{\R^N}\big[-u_1(x)\Phi_t(x,1)+\left(\mu_1 u_1(x)+u_2(x)\right)\Phi(x,1)\big]dx,
\end{array}
\end{equation}
and
\begin{equation}
\begin{array}{l}\label{energysol3-2}
\d \int_{\R^N}\big[v_t(x,t)\tilde{\Phi}(x,t)- v(x,t)\tilde{\Phi}_t(x,t)+\frac{\mu_2}{t}v(x,t) \tilde{\Phi}(x,t)\big] dx \vspace{.2cm}\\
\d \int_1^t  \int_{\R^N}v(x,s)\left[\tilde{\Phi}_{tt}(x,s)-s^{2m}\Delta \tilde{\Phi}(x,s) -\frac{\partial}{\partial s}\left(\frac{\mu_2}{s}\tilde{\Phi}(x,s)\right)+\frac{\nu^2_2}{s^2}\tilde{\Phi}(x,s)  \right]dx \,ds\vspace{.2cm}\\
\d =\int_{1}^{t}\int_{\R^N}|u_t(x,s)|^q\tilde{\Phi}(x,s)dx \, ds + \e \int_{\R^N}\big[-v_1(x)\tilde{\Phi}_t(x,1)+\left(\mu_2 v_1(x)+v_2(x)\right)\tilde{\Phi}(x,1)\big]dx.
\end{array}
\end{equation}
\end{Def}

\begin{rem}\label{rem-supp}
It is worth mentioning that it is not necessary to impose that the  test functions $\Phi$ and  $\tilde{\Phi}$ being compactly supported since the initial data $u_i$ and $v_i$ ($i=1,2$) are supported on $B_{\R^N}(0,R)$. Indeed, we have $\mbox{\rm supp}(u), \mbox{\rm supp}(v)\ \subset\{(x,t)\in\R^N\times[1,\infty): |x|\le \phi_m(t)+R\}$, where $\phi_m(t)$ is defined in \eqref{xi} below.
\end{rem}

In the following, we will state the main result in this article.

\begin{theo}
\label{blowup-equality}
Let $m\ge 0$,  $N \ge 1$, $p, q >1$, $\nu_i^2, \mu_i \ge 0$  and $\de_i \ge 0$ (for $i=1,2$)  such that 
\begin{equation}\label{assump}
\Omega(\tilde{N}_m,\mu_1,\mu_2, p,q) \ge 0,
\end{equation}
where $\de_i$,  $\Omega$ and $\tilde{N}_m$ are given, respectively, by \eqref{delta}, \eqref{blow-up-reg} and \eqref{Lambda1}.\\
Assume that  $u_1, v_1\in H^1(\R^N)$ and $u_2, v_2 \in L^2(\R^N)$ are non-negative functions which are compactly supported on  $\{(x,t)\in\R^N\times[1,\infty): |x|\le R\}$, 
do not vanish everywhere and verify
  \begin{equation}\label{hypfg}
  \frac{\mu_1-1-\sqrt{\de_1}}{2}u_1(x)+u_2(x) > 0, \quad \frac{\mu_2-1-\sqrt{\de_2}}{2}v_1(x)+v_2(x) > 0.
  \end{equation}
Let $(u,v)$ be an energy solution of \eqref{energysol2}-\eqref{energysol3} on $[1,T_\e)$ such that 
$$\mbox{\rm supp}(u), {\rm supp}(v)\ \subset \{(x,t)\in\R^N\times[1,\infty): |x|\le R+\phi_m(t)-\phi_m(1)\},$$
where 
\begin{equation}\label{xi}
\phi_{m}(t):=\frac{t^{1+m}}{1+m},
\end{equation} 
Then, there exists a constant $\e_0=\e_0(u_i,v_i,N,R,m,p,q,\mu_i,\nu_i)>0$ ($i=1,2$)
such that for any $\e \in (0,\e_0]$ the lifespan $T_\e$ verifies
\begin{equation}\label{Teps1}
T_\e \le \left\{
\begin{array}{lll}
C \e^{-\Omega(\tilde{N}_m,\mu_1,\mu_2,p,q)}& \text{if}&\Omega(\tilde{N}_m,\mu_1,\mu_2,p,q)>0,\\
\exp(C \e^{-(pq-1)})& \text{if}&\Omega(\tilde{N}_m,\mu_1,\mu_2,p,q)=0, \\
\exp(C \e^{-\min \left(\frac{pq-1}{p+1},\frac{pq-1}{q+1}\right)})& \text{if}&\Lambda(\tilde{N}_m+\mu_1,p,q)=\Lambda(\tilde{N}_m+\mu_2,q,p)=0.
\end{array}
\right.
\end{equation}
 where $C$ is a positive constant independent of $\e$,  and $\Lambda$ is given by \eqref{Lambda}.
\end{theo}

\begin{rem}
We believe that the results in Theorem \ref{blowup-equality} hold true for $-1<m<0$ and $\mu_i, \nu_i, \delta_i >0$ ($i=1,2$) even though they can be not optimal; see \cite{Tsutaya3} for more details. 
\end{rem}

\begin{rem}
As it is the case for the single equation \cite{Our3}, it is worth to mention that the blow-up results in Theorem \ref{blowup-equality}  do not depend on the mass parameters $\nu_i$ ($i=1,2$). Furthermore, one can investigate the question of optimality for the blow-up region and the upper bound of the lifespan in the spirit of Theorem \ref{blowup-equality}  to find a plausible candidate for the threshold  between the blow-up and the global existence regions.
\end{rem}

\begin{rem}
The cases $N=1,2$ present here a particularity in comparison with higher dimensions. More precisely, in dimension $1$, one can observe  a competition between the damping and Tricomi terms, in view of \eqref{Lambda1}. However, in dimension $2$, the shifted dimension $\tilde{2}_m$ does not depend on $m$ and so is $T_\e$. Starting from $N \ge 3$, the Tricomi and the damping terms  are effective in the blow-up.
\end{rem}

\section{Framework and useful tools}\label{aux}
\par

In this section, we will develop some necessary instruments related to the construction of appropriate test functions. In fact, this technique is classical now since it has been widely used for wave equations. Hence, in the same spirit of use as e.g. in \cite{Our3,Our5,HHP1,HHP2}, we define the  positive test function  $\psi_i^{\eta}(x,t)$ as the solution  of  the conjugate equation associated  with  the linear problem of \eqref{G-sys}, namely 
\begin{equation}
\label{test11}
\psi_i^{\eta}(x,t):=\rho_i^{\eta}(t)\varphi^{\eta}_i(x), \ i=1,2, \ \forall \ \eta >0,
\end{equation}
where
\begin{equation}
\label{test12}
\varphi^{\eta}_i(x):=
\left\{
\begin{array}{ll}
\d\int_{S^{N-1}}e^{\eta^{(1+m)} x\cdot\omega}d\omega & \mbox{for}\ N\ge2,\vspace{.2cm}\\
e^{\eta^{(1+m)} x}+e^{-\eta^{(1+m)} x} & \mbox{for}\  N=1.
\end{array}
\right.
\end{equation}
Note that the function $\varphi^{\eta}_i(x)$ is introduced in \cite{YZ06} and satisfies 
\begin{equation}\label{eq-phi-delta}
\Delta \varphi^{\eta}_i = \eta^{2m+2} \varphi^{\eta}_i,
\end{equation}
 and  $\rho_i^{\eta}(t)$, \cite{Palmieri1,Tu-Lin1,Tu-Lin},   is a solution of 
\begin{equation}\label{lambda}
\frac{d^2 \rho^{\eta}_i(t)}{dt^2}-\frac{d}{dt}\left(\frac{\mu_i}{t}\rho^{\eta}_i(t)\right)+\left(\frac{\nu_i^2}{t^2}-\eta^{2m+2} t^{2m}\right)\rho^{\eta}_i( t)=0, \ i=1,2.
\end{equation}
Obviously, one can see that $\rho^{\eta}_i(t)=\rho_i(\eta t)$, where $\rho_i(t):=\rho^{1}_i(t)$.

The expression of  $\rho^{\eta}_i(t)$ reads as follows:
\begin{equation}\label{lmabdaK}
\d \rho^{\eta}_i(t)=(\eta t)^{\frac{\mu_i+1}{2}}K_{\frac{\sqrt{\de_i}}{2(1+m)}}(\phi_{m}(\eta t)), \quad \forall \ t \ge 1, \ i=1,2, \ \forall \ \eta >0,
\end{equation}
where $\phi_{m}(t)$ is defined by \eqref{xi}, and 
$$K_{\nu}(t)=\int_0^\infty\exp(-t\cosh \zeta)\cosh(\nu \zeta)d\zeta,\ \nu\in \mathbb{R}.$$
Owing to the well-known properties of the modified Bessel function, one can see that
\begin{equation}\label{Knu-pp}
\frac{d}{dz}K_{\nu}(z)=-K_{\nu+1}(z)+\frac{\nu}{z}K_{\nu}(z).
\end{equation}

Now, combining \eqref{eq-phi-delta} and \eqref{lambda}, we obtain
\begin{equation}\label{lambda-eq}
\partial^2_t \psi^{\eta}_i(x, t)-t^{2m}\Delta \psi^{\eta}_i(x, t) -\frac{\partial}{\partial t}\left(\frac{\mu_i}{t}\psi^{\eta}_i(x, t)\right)+\frac{\nu_i^2}{t^2}\psi^{\eta}_i(x, t)=0.
\end{equation}

Note that the derivation of a solution to \eqref{lambda} can be found in \cite[Appendix A]{Our8}. Furthermore,   some useful properties of the function $\rho_i(t)$ are subject of the following lemma. 

\begin{lem}\label{lem-supp}
There exists a solution $\rho_i(t)$ of \eqref{lambda} which verifies the following properties:
\begin{itemize}
\item[{\bf (i)}] The function $\rho_i(t)$ is positive  on $[1,\infty)$. Furthermore,   there exists a constant $C_1$ such that   
\begin{equation}\label{est-rho}
C_1^{-1}t^{\frac{\mu_i-m}{2}} \exp(-\phi_{m}(t)) \le  \rho_i(t) \le C_1 t^{\frac{\mu_i-m}{2}} \exp(-\phi_{m}(t)), \quad \forall \ t \ge 1,\  i=1,2,
\end{equation}
where $\phi_{m}(t)$ is defined by \eqref{xi}.
\item[{\bf (ii)}] For all $m>0$, we have
\begin{equation}\label{lambda'lambda1}
\d \lim_{t \to +\infty} \frac{1}{t^{m} \rho^{\eta}_i(t)}\frac{d \rho^{\eta}_i(t)}{d t}=-\eta^{m+1},\  i=1,2.
\end{equation}
\end{itemize}
\end{lem}
\begin{proof}
See \cite[Appendix A]{Our8} for the details of the proof.
\end{proof}

Note that the notation $C$ stands for a generic positive constant depending on the data ($p,m,\mu_i,\nu_i,N,\eta,u_i,v_i,\ep_0$) but not on $\ep$.  Throughout this article, the value of $C$ may change from line to another, but we will keep the same notation for simplicity. When necessary, we will specify  the dependency of the constant $C$.\\

The following lemma  holds true for the functions $\varphi^{\eta}_i(x)$ and  $\psi^{\eta}_i(x,t)$ defined in \eqref{test11} and \eqref{test12}, receptively.
\begin{lem}[\cite{YZ06}]
\label{lem1} Let  $r>1$.
There exists a constant $C=C(m,N,\eta,\mu_i,\nu_i,R,r)>0$  such that, for all $t \ge 1$, 
\begin{equation}\label{phi-est}
\int_{|x|\leq R+\phi_{m}(t)-\phi_{m}(1)}\Big(\varphi^{\eta}_i(x)\Big)^{r}dx
\leq Ce^{r \phi_{m}(\eta t)}(1+t)^{\frac{(2-r)(N-1)(m+1)}{2}},
\quad i=1,2,
\end{equation}
\begin{equation}
\label{psi}
\int_{|x|\leq R+\phi_{m}(t)-\phi_{m}(1)}\Big(\psi^{\eta}_i(x,t)\Big)^{r}dx
\leq C (1+t)^{\frac{(2-r)(N-1)(m+1)+r(\mu_i-m)}{2}},
\quad i=1,2.
\end{equation}
\end{lem}

\par
We now introduce the following functionals which constitute essential tools to prove the blow-up criteria later on. Let
\begin{equation}
\label{F1def}
\G^{\eta} (t):=\int_{\R^N}u(x, t)\psi^{\eta}_1(x, t)dx, \quad \H^{\eta} (t):=\int_{\R^N}v(x, t)\psi^{\eta}_2(x, t)dx,
\end{equation}
and
\begin{equation}
\label{F2def}
\GG^{\eta} (t):=\int_{\R^N}\p_tu(x, t)\psi^{\eta}_1(x, t)dx, \quad \HH^{\eta} (t):=\int_{\R^N}\p_t v(x, t)\psi^{\eta}_2(x, t)dx.
\end{equation}

In what follows, we will show that the functionals  $G^{\eta}_i(t)$ and $\tilde{G}^{\eta}_i(t)$ ($i=1,2$) are bounded from below by $C\e \, t^{-m}$ and $C\e$, respectively.  
\begin{lem}
\label{F1}
Let $(u,v)$ be an energy solution of  \eqref{G-sys} in the sense of Definition \ref{def1}. Assume the hypotheses of Theorem \ref{blowup-equality} hold true. Then,  we have
\begin{equation}
\label{F1postive}
G^{\eta}_i(t)\ge C_{G^{\eta}_i}\, \e \, t^{-m}, 
\quad \forall \ t \ge 1,  \ \forall \ \eta > 0, \quad i=1,2,
\end{equation}
where $C_{G^{\eta}_i}$ is a positive constant which may depend on $\mu_i,\nu_i,N,\eta,u_i,v_i,\ep_0, R,m$.
\end{lem}

\begin{proof} 
First, note that the proof will be done just for $G^{\eta}_1(t)$. The same proof follows for $G^{\eta}_2(t)$.\\
Let $t \in [1,T)$. Using \eqref{energysol2-2} while substituting $\Phi(x, t)$ by $\psi^{\eta}_1(x, t)$, and employing \eqref{test11} and \eqref{lambda-eq}, we obtain
\begin{equation}
\begin{array}{l}\label{eq5}
\d \int_{\R^N}\big[u_t(x,t)\psi^{\eta}_1(x,t)- u(x,t)\partial_t \psi^{\eta}_1(x,t)+\frac{\mu_1}{t}u(x,t) \psi^{\eta}_1(x,t)\big]dx
\vspace{.2cm}\\
\d=\int_1^t\int_{\R^N}|v_t(x,s)|^p\psi^{\eta}_1(x,s)dx \, ds 
+\d \e \, C_1(\eta, u_1,u_2),
\end{array}
\end{equation}
where 
\begin{equation}\label{Cfg}
C_1(\eta, u_1,u_2):=\int_{\R^N}\big[\big(\mu_1\rho^{\eta}_1(1)-(\rho^{\eta}_1)'(1)\big)u_0(x)+\rho^{\eta}_1(1)u_1(x)\big] \varphi^{\eta}(x)dx.
\end{equation}
Thanks to \eqref{lmabdaK} and \eqref{Knu-pp}, we infer that
\begin{equation}\label{Cfg1}
\mu_1\rho^{\eta}_1(1)-(\rho^{\eta}_1)'(1)=\frac{\mu_1-1-\sqrt{\de_1}}{2}K_{\frac{\sqrt{\de_1}}{2(1+m)}}(\phi_{m}(\eta))+K_{\frac{\sqrt{\de_1}}{2(1+m)} +1}(\phi_{m}(\eta)),
\end{equation}
which implies that 
\begin{align}\label{Cfg}
C_1(\eta, u_1,u_2)&=K_{\frac{\sqrt{\de_1}}{2(1+m)}}(\phi_{m}(\eta)) \int_{\R^N} \big[\frac{\mu_1-1-\sqrt{\de_1}}2u_1(x)+u_2(x)\big]\varphi^{\eta}(x)dx \\&\ +K_{\frac{\sqrt{\de_1}}{2(1+m)} +1}(\phi_{m}(\eta))\int_{\R^N}u_1(x)\varphi^{\eta}(x)dx. \nonumber
\end{align}
In view of \eqref{hypfg},  the constant $C_1(\eta, u_1,u_2)$ is clearly positive.\\
By the definition of $G^{\eta}_1(t)$, given by \eqref{F1def},  and \eqref{test11}, the identity \eqref{eq5} yields
\begin{equation}
\begin{array}{l}\label{eq6}
\d \left(G_1^{\eta}\right)'(t)+\Gamma_1^{\eta}(t)G^{\eta}_1(t)=\int_1^t\int_{\R^N}|v_t(x,s)|^p\psi^{\eta}_1(x,s)dx \, ds +\e \, C_1(\eta, u_1,u_2),
\end{array}
\end{equation}
where 
\begin{equation}\label{gamma}
\d \Gamma_i^{\eta}(t):=\frac{\mu_i}{t}- \frac{2}{\rho^{\eta}_i(t)}\frac{d\rho^{\eta}_i(t)}{dt}, \quad i=1,2.
\end{equation}
Multiplying  \eqref{eq6} by $\d \frac{t^{\mu_1}}{(\rho^{\eta}_1(t))^2}$ and integrating over $(1,t)$, we obtain
\begin{align}\label{est-G1}
 G^{\eta}_1(t)
\ge G^{\eta}_1(1)\frac{(\rho^{\eta}_1(t))^2}{t^{\mu_1}}+{\e}C_1(\eta, u_0,u_1)\frac{(\rho^{\eta}_1(t))^2}{t^{\mu_1}}\int_1^t\frac{s^{\mu_1}}{(\rho^{\eta}_1(s))^2}ds.
\end{align}
Thanks to the fact that $G^{\eta}_1(1)$ is positive and using \eqref{lmabdaK},  the estimate \eqref{est-G1} gives
\begin{align}\label{est-G1-1}
 G^{\eta}_1(t)
\ge {\e}C_1(\eta, u_1,u_2)\frac{(\rho^{\eta}_1(t))^2}{t^{\mu_1}}\int_{1}^t\frac{s^{\mu_1}}{(\rho^{\eta}_1(s))^2}ds, \quad \forall \ t \ge 1.
\end{align}
Employing Lemma \ref{lem-supp} in \eqref{est-G1-1}, and recalling the definition of $\phi_{m}(t)$, given by \eqref{xi},   we obtain
\begin{align}\label{est-U-2}
 G^{\eta}_1(t)
&\ge \e C t^{-m}e^{-2\phi_{m}(\eta t)}\int^t_{1}\phi_{m}'(\eta s)e^{2\phi_{m}(\eta s)}ds, \ \forall \ t \ge 1.
\end{align}
The above inequality implies that
\begin{align}\label{est-G1-3}
 G^{\eta}_1(t)
\ge \e  C t^{-m}, \ \forall \ t \ge 1.
\end{align}

The proof of Lemma \ref{F1} is finished.
\end{proof}

In the next lemma, we will derive a lower bound for the functional $\tilde{G}_i(t)$ which depends on $\ep$, the initial data size.

\begin{lem}\label{F11}
Let $(u,v)$ be an energy solution of  \eqref{G-sys} in the sense of Definition \ref{def1}. Assume the hypotheses of Theorem \ref{blowup-equality} hold true. Then,  for $i=1,2$, there exists $\eta_0>0$ such that
\begin{equation}
\label{F1postive}
\tilde{G}^{\eta}_i(t)\ge C_{\tilde{G}^{\eta}_i}\, \e, 
\quad \forall \ t \ge1, \ \forall \  \eta  \ge \eta_0, \quad i=1,2, 
\end{equation}
where $C_{\tilde{G}^{\eta}_i}$ is a positive constant which may depend on $\mu_i,\nu_i,N,\eta,u_i,v_i,\ep_0, R,m$.
\end{lem}
\begin{proof}
We will prove the lemma only for $\tilde{G}^{\eta}_1$. The same conlcusion holds for $\tilde{G}^{\eta}_2$ with the necessary modifications. \\
Let $t \in [1,T)$. Thanks to the definition of ${G}^{\eta}_1$ and  $\tilde{G}^{\eta}_1$, given  by \eqref{F1def} and  \eqref{F2def}, respectively, and using \eqref{test11} together with the following identity
 \begin{equation}\label{def23}\d \frac{d G^{\eta}_1}{dt}(t) -\frac{1}{\rho^{\eta}_1(t)}\frac{d\rho^{\eta}_1(t)}{dt}G^{\eta}_1(t)= \tilde{G}^{\eta}_1(t),\end{equation}
 the equation  \eqref{eq6} gives
\begin{equation}
\begin{array}{l}\label{eq5bis}
\d \tilde{G}^{\eta}_1(t)+\left(\frac{\mu_1}{t}-\frac{1}{\rho^{\eta}_1(t)}\frac{d\rho^{\eta}_1(t)}{dt}\right)G^{\eta}_1(t)\\
=\d \int_1^t\int_{\R^N}|v_t(x,s)|^p\psi^{\eta}_1(x,s)dx \, ds +\e \, C_1(\eta, u_1,u_2).
\end{array}
\end{equation}
Now, differentiating in time the  equation \eqref{eq5bis}, we infer that
\begin{equation}\label{F1+bis}
\begin{array}{l}
\d \frac{d \tilde{G}^{\eta}_1}{dt}(t)+\left(\frac{\mu_1}{t}-\frac{1}{\rho^{\eta}_1(t)}\frac{d\rho^{\eta}_1(t)}{dt}\right)\frac{d G^{\eta}_1}{dt}(t)\\
\d -\left(\frac{\mu_1}{t^2}+\frac{1}{\rho^{\eta}_1(t)}\frac{d^2\rho^{\eta}_1(t)}{dt^2}-\frac{1}{(\rho^{\eta}_1(t))^2}\left(\frac{d\rho^{\eta}_1(t)}{dt}\right)^2 \right)G^{\eta}_1(t) 
\d = \int_{\R^N}|v_t(x,t)|^p\psi^{\eta}_1(x,t)dx.
\end{array}
\end{equation}
Employing  \eqref{lambda} and   \eqref{def23}, the equation  \eqref{F1+bis} yields
\begin{equation}\label{F1+bis2}
\begin{array}{l}
\d \frac{d \tilde{G}^{\eta}_1}{dt}(t)+\left(\frac{\mu_1}{t}-\frac{1}{\rho^{\eta}_1(t)}\frac{d\rho^{\eta}_1(t)}{dt}\right)\tilde{G}^{\eta}_1(t)\\
\d -\left(\eta^{2+2m} t^{2 m}-\frac{\nu_1^2}{t^2}\right)G^{\eta}_1(t) 
\d = \int_{\R^N}|v_t(x,t)|^p\psi^{\eta}_1(x,t)dx.
\end{array}
\end{equation}
Using the definition of $\Gamma_1^{\eta}(t)$, \eqref{gamma},   we conclude that 
\begin{equation}\label{G2+bis3}
\begin{array}{c}
\d \frac{d \tilde{G}^{\eta}_1}{dt}(t)+\frac{3\Gamma_1^{\eta}(t)}{4}\tilde{G}^{\eta}_1(t)\ge\Sigma_1^1(t)+\Sigma_1^2(t)+\Sigma_1^3(t),
\end{array}
\end{equation}
where 
\begin{equation}\label{sigma1-exp}
\begin{array}{rl}
\Sigma_1^1(t):=&\d \left(-\frac{1}{2\rho^{\eta}_1(t)}\frac{d\rho^{\eta}_1(t)}{dt}-\frac{\mu_1}{4t}\right)\left(\tilde{G}^{\eta}_1(t)+\left(\frac{\mu_1}{t}-\frac{1}{\rho^{\eta}_1(t)}\frac{d\rho^{\eta}_1(t)}{dt}\right)G^{\eta}_1(t)\right),
\end{array} 
\end{equation}
\begin{equation}\label{sigma2-exp}
\Sigma_1^2(t):=\d \left(\eta^{2+2m} t^{2 m}-\frac{\nu_1^2}{t^2}+\left(\frac{1}{2\rho^{\eta}_1(t)}\frac{d\rho^{\eta}_1(t)}{dt}+\frac{\mu_1}{4t}\right) \left(\frac{\mu_1}{t}-\frac{1}{\rho^{\eta}_1(t)}\frac{d\rho^{\eta}_1(t)}{dt}\right) \right)  G^{\eta}_1(t),
\end{equation}
and
\begin{equation}\label{sigma3-exp}
\Sigma_1^3(t):=  \int_{\R^N}|v_t(x,t)|^p\psi^{\eta}_1(x,t) dx.
\end{equation}
In the same spirit as in \eqref{lambda'lambda1}, but for $t=1$ and as $\eta \to \infty$, we have
\begin{equation}\label{lambda'lambda2}
\d \lim_{\eta \to +\infty} \left(\frac{-1}{\eta^{m+1}\rho^{\eta}_1(1)}\frac{d\rho^{\eta}_1(1)}{dt}\right)=1.
\end{equation}
Hence, there exists $\eta_1>0$ (large enough) such that
\begin{equation}\label{lambda'lambda3}
\d \frac{-1}{\eta^{m+1}\rho^{\eta}_1(1)}\frac{d\rho^{\eta}_1(1)}{dt} \ge \frac34, \quad \forall \ \eta > \eta_1.
\end{equation}
Remember the definition of $\rho^{\eta}_1(t)$, given by \eqref{lmabdaK}, we observe that
\begin{equation}\label{lambda'lambda4}
\d -\frac{1}{2\rho^{\eta}_1(t)}\frac{d\rho^{\eta}_1(t)}{dt} = -\frac{1}{2 t \rho^{\eta t}_1(1)}\frac{d\rho^{\eta t}_1(1)}{dt} 
 \ge \frac{3\eta^{m+1} t^{m}}8, \quad \forall \ \eta > \eta_1, \ \forall \ t \ge 1.
\end{equation}
Therefore, since $m \ge 0$, there exists $\eta_2>\eta_1$ such that
\begin{equation}\label{lambda'lambda5}
-\frac{1}{2\rho^{\eta}_1(t)}\frac{d\rho^{\eta}_1(t)}{dt}-\frac{\mu_1}{4t}
 \ge  \frac{\eta^{m+1} t^{m}}4, \quad \forall \ \eta > \eta_2, \ \forall \ t \ge 1.
\end{equation}

Now, inserting  \eqref{eq5bis} and \eqref{lambda'lambda5} in \eqref{sigma1-exp}, we obtain 
\begin{equation}\label{sigma1}
\d \Sigma_1^1(t) \ge  \frac{\eta^{m+1} t^{m} \e}{4} \, C_1(\eta, u_1,u_2)+\frac{\eta^{m+1} t^{m}}{4}\int_{1}^t \int_{\R^N}|v_t(x,s)|^p\psi_1^{\eta}(x,s)dx ds, \quad \forall \ t \ge 1. 
\end{equation}
Employing Lemma \ref{F1} and \eqref{lambda'lambda5}, one can obtain the existence of  $ \eta_3>\eta_2$ for which we have
\begin{equation}\label{sigma2}
\d \Sigma_1^2(t) \ge 0, \quad \forall \ \eta \ge  \eta_3, \quad \forall \ t \ge 1. 
\end{equation}
Gathering all the above results, namely \eqref{G2+bis3}, \eqref{sigma3-exp}, \eqref{sigma1} and \eqref{sigma2}, we end up with the following estimate
\begin{equation}\label{G2+bis41}
\begin{array}{rcl}
\d \frac{d \tilde{G}^{\eta}_1}{dt}(t)+\frac{3\Gamma_1^{\eta}(t)}{4}\tilde{G}^{\eta}_1(t) &\ge& \d \frac{\eta^{m+1} t^{m} \e}{4} \, C_1(\eta, u_1,u_2)+\frac{\eta^{m+1} t^{m}}{4}\int_{1}^t \int_{\R^N}|v_t(x,s)|^p\psi_1^{\eta}(x,s)dx ds \\ &+&\d \int_{\R^N}|v_t(x,t)|^p\psi_1^{\eta}(x,t) dx, \quad \forall \ \eta \ge  \eta_3, \quad \forall \ t \ge 1. 
\end{array}
\end{equation}
At this level, we can eliminate  the nonlinear term in the above estimate,  then we write 
\begin{equation}\label{G2+bis4}
\begin{array}{rcl}
\d \frac{d \tilde{G}^{\eta}_1}{dt}(t)+\frac{3\Gamma_1^{\eta}(t)}{4}\tilde{G}^{\eta}_1(t) &\ge& \d \frac{\eta^{m+1} t^{m} \e}{4} \, C_1(\eta, u_1,u_2), \quad \forall \ \eta \ge  \eta_3, \quad \forall \ t \ge 1. 
\end{array}
\end{equation}
Integrating \eqref{G2+bis4} over $(1,t)$ after multiplication by $\d \frac{t^{3\mu_1/4}}{\left(\rho^{\eta}_1(t)\right)^{3/2}}$, we obtain
\begin{align}\label{est-G111-new}
 \tilde{G}^{\eta}_1(t)
\ge \frac{\tilde{G}^{\eta}_1(1)}{\left(\rho^{\eta}_1(1)\right)^{3/2}}\frac{\left(\rho^{\eta}_1(t)\right)^{3/2}}{t^{3\mu_1/4}}+\frac{\eta^{m+1} \e}{4} \, C_1(\eta, u_1,u_2)\frac{\left(\rho^{\eta}_1(t)\right)^{3/2}}{t^{3\mu_1/4}}
\int_{1}^t\frac{s^{3\mu_1/4+m}}{\left(\rho^{\eta}_1(s)\right)^{3/2}}ds, \quad \forall \ t  \ge  1. 
\end{align}
Remembering that $\tilde{G}^{\eta}_1(1) \ge 0$, \eqref{est-G111-new} yields
\begin{align}\label{est-G111-new1}
 \tilde{G}^{\eta}_1(t)
\ge \frac{\eta^{m+1} \e}{4} \, C_1(\eta, u_1,u_2)\frac{\left(\rho^{\eta}_1(t)\right)^{3/2}}{t^{3\mu_1/4}}
\int_{1}^t\frac{s^{3\mu_1/4+m}}{\left(\rho^{\eta}_1(s)\right)^{3/2}}ds, \quad \forall \ \eta \ge  \eta_3, \quad \forall \ t  \ge  1. 
\end{align}
Employing Lemma \ref{lem-supp} in \eqref{est-G111-new1}, and recalling the definition of $\phi_{m}(t)$, given by \eqref{xi},   we obtain
\begin{align}\label{est-G111-new2}
 \tilde{G}^{\eta}_1(t)
&\ge  \e \, C t^{-m}e^{-\frac32\phi_{m}(\eta t)}\int^t_{1}s^{m}\phi_{m}'(\eta s)e^{\frac32 \phi_{m}(\eta s)}ds, \quad \forall \ \eta \ge  \eta_3, \ \forall \ t \ge 1.
\end{align}
The above inequality implies that
\begin{equation}\label{est-G2-12}
 \tilde{G}^{\eta}_1(t)
\ge   \e \, C e^{-\frac32\phi_{m}(\eta t)}\int^t_{\sup(1,t/2)}\phi_{m}'(\eta s)e^{\frac32 \phi_{m}(\eta s)}ds, \quad \forall \ \eta \ge  \eta_3, \ \forall \ t \ge 1.
\end{equation}
Consequently, we obtain that
\begin{align}\label{est-G1-2}
 \tilde{G}^{\eta}_1(t)
\ge  C\,{\e}, \quad \forall \ \eta \ge  \eta_0:=\eta_3, \quad \forall \ t \ge 1.
\end{align}

This concludes the proof of Lemma \ref{F11}.
\end{proof}


\section{Proof of Theorem \ref{blowup-equality}.}\label{sec-ut}

First,  we introduce  the following functionals:
\begin{equation}\label{L1}
L^{\eta}_1(t):=
\frac{1}{8}\int_{1}^t  \int_{\R^N}|v_t(x,s)|^p\psi^{\eta}_1(x,s)dx ds
+\frac{C_3 \e}{8},
\end{equation}
and
\begin{equation}\label{L2}
L^{\eta}_2(t):=
\frac{1}{8}\int_{1}^t  \int_{\R^N}|u_t(x,s)|^q\psi^{\eta}_2(x,s)dx ds
+\frac{C_3 \e}{8},
\end{equation}
where $C_3=\min(C_1(\eta,u_1,u_2)/4,C_2(\eta, v_1,v_2)/4,8C_{\tilde{G}^{\eta}_1},8C_{\tilde{G}^{\eta}_2})$ (see Lemma  \ref{F11} for the constants $C_{\tilde{G}^{\eta}_1}$ and $C_{\tilde{G}^{\eta}_2}$).

As sated in \eqref{lambda'lambda1}, but for $t=1$ and for $\eta \to \infty$, we infer that
\begin{equation}\label{lambda'lambda2}
\d \lim_{\eta \to +\infty} \left(\frac{-1}{\eta^{m+1}\rho^{\eta}_1(1)}\frac{d\rho^{\eta}_1(1)}{dt}\right)=1.
\end{equation}
Hence, there exists $\eta_4>\eta_0$ (large enough) such that
\begin{equation}\label{lambda'lambda3}
\d \frac98 \ge \frac{-1}{\eta^{m+1}\rho^{\eta}_1(1)}\frac{d\rho^{\eta}_1(1)}{dt} \ge \frac34, \quad \forall \ \eta \ge \eta_4.
\end{equation}
Remember the definition of $\rho^{\eta}_1(t)$, given by \eqref{lmabdaK}, we observe that
\begin{equation}\label{lambda'lambda4}
\d \frac{9\eta^{m+1} t^{m}}8 \ge -\frac{1}{\rho^{\eta}_1(t)}\frac{d\rho^{\eta}_1(t)}{dt} = -\frac{1}{t \rho^{\eta t}_1(1)}\frac{d\rho^{\eta t}_1(1)}{dt} 
 \ge \frac{3\eta^{m+1} t^{m}}4, \quad \forall \ \eta \ge \eta_4, \ \forall \ t \ge 1.
\end{equation}
Therefore, since $m \ge 0$, we have
\begin{equation}\label{lambda'lambda6}
\frac{9 \eta^{m+1} t^{m}}8 \ge -\frac{1}{\rho^{\eta}_1(t)}\frac{d\rho^{\eta}_1(t)}{dt}
 \ge  \frac{3 \eta^{m+1} t^{m}}4, \quad \forall \ \eta \ge \eta_4, \ \forall \ t \ge 1.
\end{equation}
In view of \eqref{lambda'lambda6}, it is easy to see that  
\begin{equation}\label{Gamma-pos}
\d \Gamma_i^{\eta}(t):=\frac{\mu_i}{t}- \frac{2}{\rho^{\eta}_i(t)}\frac{d\rho^{\eta}_i(t)}{dt} > 0, \  i=1,2, \quad \forall \ \eta \ge \eta_4, \ \forall \ t \ge 1.
\end{equation}
Moreover, we have $$\Gamma_i^{\eta}(t) \le \frac{9 \eta^{m+1} t^{m}}4 +\frac{\mu_i}{t}, \quad \forall \ \eta \ge \eta_4, \ \forall \ t \ge 1.$$
Note that there exists $\eta_5 > \eta_4$ large enough such that  
$$\d \frac{\eta^{m+1}t^{m}}{4}-\frac{3}{32} \left( \frac{9 \eta^{m+1} t^{m}}4 +\frac{\mu_i}{t} \right)>0, \  i=1,2,   \ \forall  \ t \ge 1,   \ \forall  \  \eta \ge \eta_5.$$
This yields
\begin{equation}\label{Gamma-pos-diff}
\d \frac{\eta^{m+1}t^{m}}{4}-\frac{3  \Gamma_i^{\eta}(t)}{32}>0, \  i=1,2,   \ \forall  \ t \ge 1,   \ \forall  \  \eta \ge \eta_5.
\end{equation}
From  now, we set $\eta =  \eta_5$, but we will continue in the subsequent to keep the superscript notation using $\eta$ in order to simplify the presentation.

Now, let us define the functional
$$\mathcal{F}^{\eta}_i(t):= \tilde{G}^{\eta}_i(t)-L^{\eta}_i(t), \quad \forall \ i=1,2.$$
Hence, thanks to \eqref{G2+bis41}, we see that $\mathcal{F}^{\eta}_1$ satisfies
\begin{equation}\label{G2+bis6}
\begin{array}{rcl}
\d \frac{d \mathcal{F}^{\eta}_1(t)}{dt}+\frac{3\Gamma_1^{\eta}(t)}{4}\mathcal{F}^{\eta}_1(t) &\ge& \d \left(\frac{\eta^{m+1}t^{m}}{4}-\frac{3\Gamma_1^{\eta}(t)}{32}\right)\int_{1}^t \int_{\R^N}|v_t(x,s)|^p\psi^{\eta}_1(x,s)dx ds\vspace{.2cm}\\ &+&  \d \frac{7}{8}\int_{\R^N}|v_t(x,t)|^p\psi^{\eta}_1(x,t) dx+C_3 \left(\frac{\eta^{m+1}t^{m}}{4}-\frac{3\Gamma_1^{\eta}(t)}{32}\right) \e\\
&\ge&0, \quad \forall \ t \ge 1.
\end{array}
\end{equation}
Integrating \eqref{G2+bis6} on $(1,t)$ after multiplication  by $\frac{t^{3 \mu_1/4}}{(\rho_1^{\eta}(t))^{3/2}}$ yields
\begin{align}\label{est-G111}
 \mathcal{F}_1^{\eta}(t)
\ge \mathcal{F}_1^{\eta}(1)\frac{(\rho_1^{\eta}(t))^{3/2}}{t^{3 \mu_1/4}(\rho_1^{\eta}(1))^{3/2}}, \quad \forall \ t \ge 1,
\end{align}
where $\rho_1^{\eta}(t)$ is defined by \eqref{lmabdaK}.\\
Using Lemma \ref{F11} and $C_3=\min(C_1(\eta,u_1,u_2)/4,C_2(\eta, v_1,v_2)/4,8C_{\tilde{G}^{\eta}_1},8C_{\tilde{G}^{\eta}_2}) \le 8C_{\tilde{G}^{\eta}_1}$, one can see that $\d \mathcal{F}^{\eta}_1(1)=\tilde{G}^{\eta}_1(1)-\frac{C_3 \e}{8} \ge C_{\tilde{G}_1}\e -\frac{C_3 \e}{8}\ge 0$. \\
Consequently, we infer that
\begin{equation}
\label{G2-est}
\tilde{G}^{\eta}_1(t) \geq L^{\eta}_1(t), \quad \forall \ t \ge 1.
\end{equation}
In a similar way, we have an analogous lower bound for $\tilde{G}^{\eta}_2(t)$, that is
\begin{equation}
\label{G2-est-bis}
\tilde{G}^{\eta}_2(t) \geq L^{\eta}_2(t), \quad \forall \ t \ge 1.
\end{equation}
Thanks to  H\"{o}lder's inequality,  one can see that
\begin{equation}
\begin{array}{rcl}
\d (\tilde{G}^{\eta}_2(t))^p &\le&  \d \int_{\R^N}|v_t(x,t)|^p\psi^{\eta}_1(x,t)dx \left(\int_{|x|\leq R+\phi_{m}(t)-\phi_{m}(1)}(\psi^{\eta}_2(x,t))^{\frac{p}{p-1}}(\psi^{\eta}_1(x,t))^{\frac{-1}{p-1}}dx\right)^{p-1}\\
&\le& \d \int_{\R^N}|v_t(x,t)|^p\psi^{\eta}_1(x,t)dx  (\rho^{\eta}_2(t))^{p} (\rho^{\eta}_1(t))^{-1} \left(\int_{|x|\leq R+\phi_{m}(t)-\phi_{m}(1)}\varphi^{\eta}_1(x) dx\right)^{p-1}.
\end{array}
\end{equation}
Employing \eqref{phi-est}  (with $r=1$) and \eqref{est-rho}, we obtain
\begin{equation}\label{new-4.10}
\d (\tilde{G}^{\eta}_2(t))^p \le \d C t^{\frac{(p-1)(N-1)(m+1)+p(\mu_2-m)-(\mu_1-m)}{2}}\int_{\R^N}|v_t(x,t)|^p\psi^{\eta}_1(x,t) dx.
\end{equation}
Combining  \eqref{G2-est-bis} and \eqref{new-4.10}, we conclude that
\begin{equation}
\label{inequalityfornonlinearin}
\d \frac{d L^{\eta}_1(t)}{dt}\geq C (L^{\eta}_2(t))^p  t^{-\frac{(p-1)[(N-1)(m+1)-m]}{2}+ \frac{\mu_1}{2}-\frac{\mu_2}{2}p}, \ \forall \ t \ge 1.
\end{equation}
 Likewise, we have
\begin{equation}
\label{inequalityfornonlinearin2}
\frac{d L^{\eta}_2(t)}{dt}\geq C (L^{\eta}_1(t))^q  t^{-\frac{(q-1)[(N-1)(m+1)-m]}{2}+ \frac{\mu_2}{2}-\frac{\mu_1}{2}q}, \ \forall \ t \ge 1.
\end{equation}

A straightforward integration of \eqref{inequalityfornonlinearin} and \eqref{inequalityfornonlinearin2} on $(1, t)$ yields, respectively,
\begin{equation}\label{integ-ineq}
L^{\eta}_1(t)\geq \frac{C_3 \e}{8}+C \int_{1}^t  s^{-\frac{(N-1)(m+1)-m}{2}(p-1)+ \frac{\mu_1}{2}-\frac{\mu_2}{2}p}(L^{\eta}_2(s))^p ds, \quad \forall \ t \ge 1,
\end{equation}
and
\begin{equation}\label{integ-ineq2}
L^{\eta}_2(t)\geq \frac{C_3 \e}{8}+C \int_{1}^t  s^{-\frac{(N-1)(m+1)-m}{2}(q-1)+ \frac{\mu_2}{2}-\frac{\mu_1}{2}q}(L^{\eta}_1(s))^q ds, \quad \forall \ t \ge 1.
\end{equation}

Finally, the remaining part of the proof can be mimicked line by line from the one in \cite[Sections 4.2 and 4.3]{Palmieri}. For that purpose, some necessary modifications are necessary; for example here \eqref{integ-ineq} (resp. \eqref{integ-ineq2} corresponds to (25) (resp. (26)) in \cite{Palmieri}. The main difference between the two results consists in the fact that here the shift of the dimension $N$ is with  $\mu_i$ instead of $\sigma(\mu_i)$, a function of $\mu_i$, in \cite{Palmieri}.

 The proof of Theorem \ref{blowup-equality} is thus completed.

\section{Concluding remarks and some open questions}\label{open}
We derived in this article a blowing-up result related to a system composed of two coupled Euler-Poisson-Darboux-Tricomi equations with time derivative nonlinearity. Furthermore, we showed an upper bound of the lifespan time of existence with respect to the the initial data size which supposed to be small throughout all this work. In this direction, some natural extensions can be investigated. For example, the sharpness of the lifespan estimate constitutes an open question that can be studied. Moreover, the case of different speeds of propagation and the case of combined nonlinearities will be considered elsewhere. Finally, an interesting question  would be the consideration of the same system subject to a space time dependent damping with or without a fractional dissipation; see e.g. \cite{Kirane1, Kirane2}.

\appendix \section{} \label{appendix1}

In this appendix, we will perform some numerical tests on the the functions $Y_1(t)$ and $Y_2(t)$ inherited from the solutions to the coupled system of differential inequalities \eqref{inequalityfornonlinearin}-\eqref{inequalityfornonlinearin2}. It is worth mentioning that the numerical study of \eqref{inequalityfornonlinearin}-\eqref{inequalityfornonlinearin2} could be somehow challenging and consequently this is out of the scope of this article. To be precise,  let $Y_1(t)$ and $Y_2(t)$ be the solutions of the following system:
\begin{equation}\label{new-Y}
\left\{\begin{array}{l}
\d \frac{d Y_1(t)}{dt} = (Y_2(t))^p  t^{-\frac{(p-1)[(N-1)(m+1)-m]}{2}+ \frac{\mu_1}{2}-\frac{\mu_2}{2}p}, \ \forall \ t \ge 1,\vspace{.2cm}\\
\d \frac{d Y_2(t)}{dt} = (Y_1(t))^q  t^{-\frac{(q-1)[(N-1)(m+1)-m]}{2}+ \frac{\mu_2}{2}-\frac{\mu_1}{2}q}, \ \forall \ t \ge 1.
\end{array}
\right.
\end{equation}
where the $\eta-$dependency notation is ignored for simplicity.\\
In view of the definitions \eqref{L1} and  \eqref{L2}, we associate with \eqref{new-Y} the following initial data:
\begin{equation}\label{new-Y-IC}
Y_1(t=1)=Y_2(t=1)= \varepsilon.
\end{equation}

It is noticeable that the  system \eqref{new-Y} is somehow a borderline of the system  \eqref{inequalityfornonlinearin}-\eqref{inequalityfornonlinearin2}. Nevertheless, the numerical study of  \eqref{new-Y}  will give us enough information about the blow-up phenomena;  this will be the main objective of the appendix. Indeed, as it will be evident later  in the simulation results, starting from small positive initial data as in \eqref{new-Y-IC}, the solutions $Y_1(t)$ and $Y_2(t)$ blow up in a finite time according to the different cases that will be considered.

Performing the numerical simulation of \eqref{new-Y}-\eqref{new-Y-IC} furnishes the approximate solutions of  $Y_1(t)$ and $Y_2(t)$  as shown in Figure \ref{figures}. Let us denote $T_b$ (resp. $T_f$) the numerical blow-up (resp. final) time. We will give an estimate of the value of $T_b$ in each case under investigation in the numerical simulations.

\begin{figure}[htb]
    \centering 
\begin{subfigure}{0.25\textwidth}
  \includegraphics[width=\linewidth]{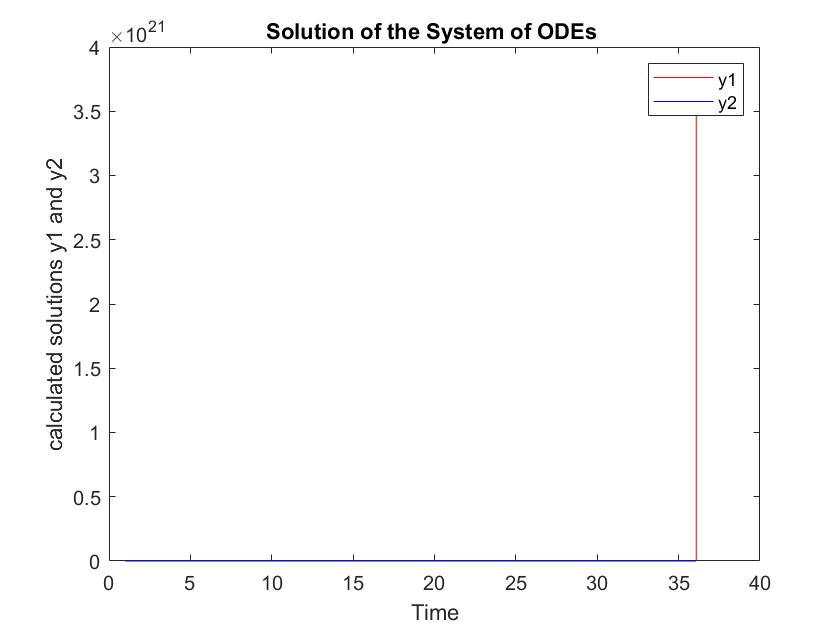}
  \caption{\tiny{The case $p=2$, $q=1.5$, $m=1$, $N=1$, $\mu_1=4$, $\mu_2=2$ and  $\varepsilon =0.1$ ($\Omega=0.75$ and $T_b\approx 37$).}}  
  \label{fig1}
\end{subfigure}\hfil 
\begin{subfigure}{0.25\textwidth}
  \includegraphics[width=\linewidth]{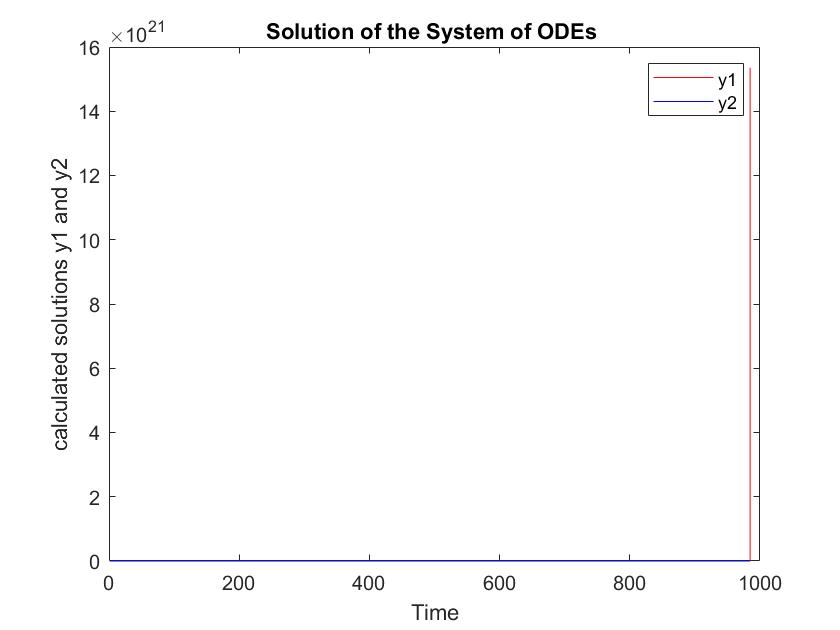}
 \caption{\tiny{The case $p=2, q=1.5, m=1, N=1, \mu_1=4, \mu_2=2$ and $\varepsilon =0.01$ ($\Omega=0.75$ and $T_b\approx 990$).}} 
  \label{fig2}
\end{subfigure}\hfil 
\begin{subfigure}{0.25\textwidth}
  \includegraphics[width=\linewidth]{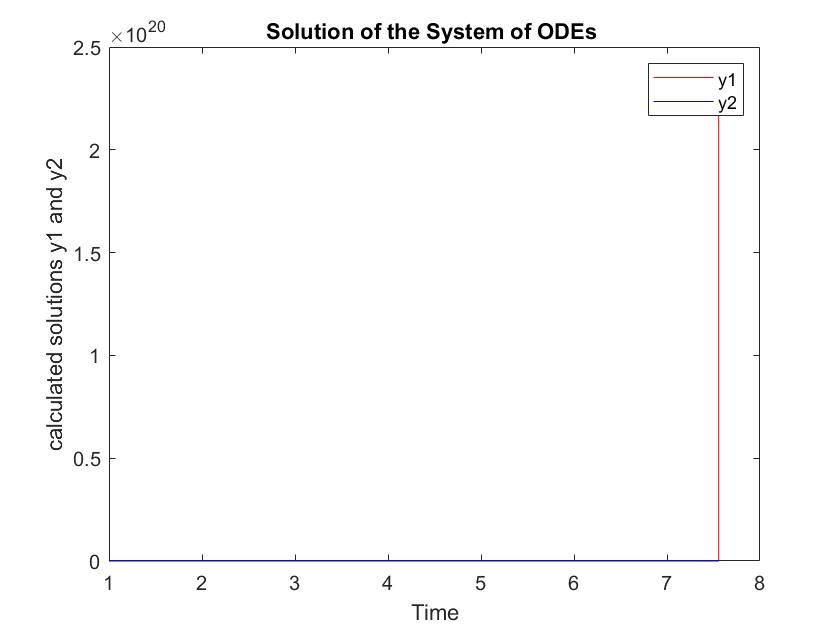}
   \caption{\tiny{The case $p=2, q=1.5, m=1, N=1, \mu_1=0, \mu_2=2$ and  $\varepsilon =0.1$ ($\Omega=2$ and $T_b\approx 8$).}} 
  \label{fig3}
\end{subfigure}

\medskip
\begin{subfigure}{0.25\textwidth}
  \includegraphics[width=\linewidth]{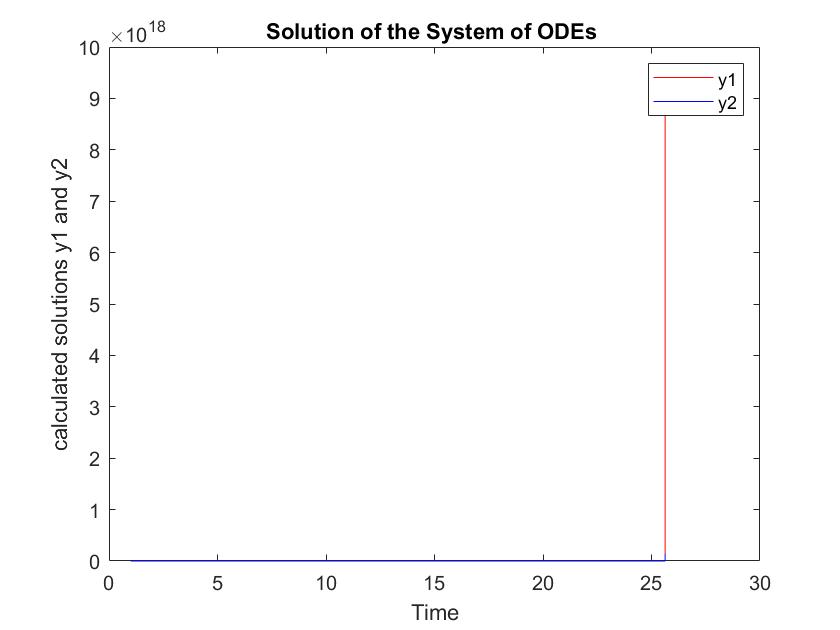}
  \caption{\tiny{The case $p=2, q=1.5, m=1, N=1, \mu_1=0, \mu_2=2$ and  $\varepsilon =0.01$ ($\Omega=2$ and $T_b\approx 26$).}}
  \label{fig4}
\end{subfigure}\hfil 
\begin{subfigure}{0.25\textwidth}
  \includegraphics[width=\linewidth]{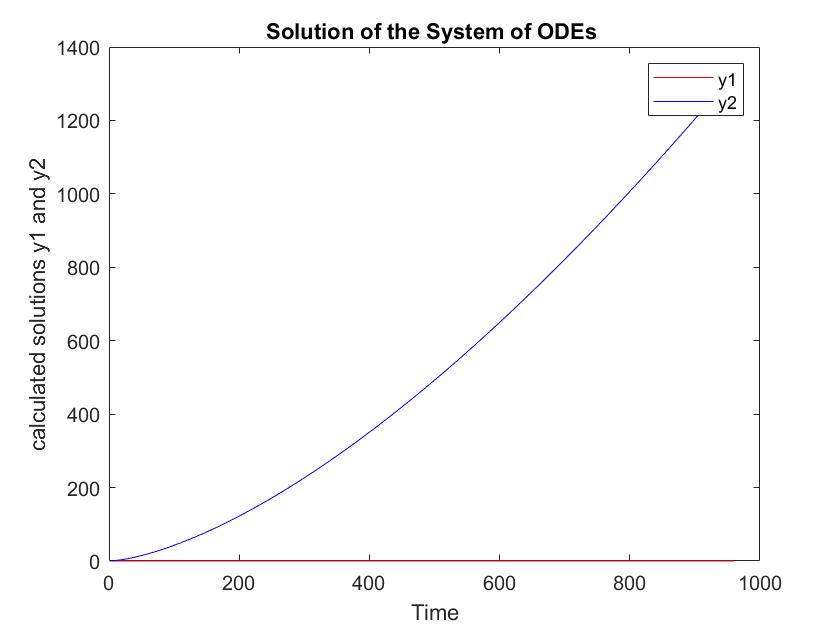}
  \caption{\tiny{The case $p=2, q=1.25, m=1, N=2, \mu_1=3, \mu_2=5$ and  $\varepsilon =0.1$ ($\Omega=0$ and $T_f\approx 960$).}}
  \label{fig5}
\end{subfigure}\hfil 
\begin{subfigure}{0.25\textwidth}
  \includegraphics[width=\linewidth]{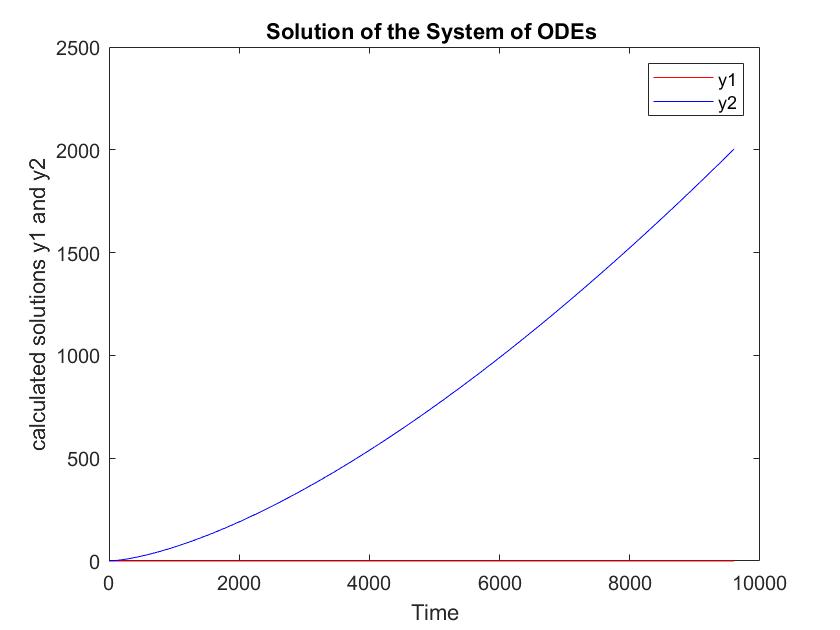}
  \caption{\tiny{The case $p=2, q=1.25, m=1, N=2, \mu_1=3, \mu_2=5$ and  $\varepsilon =0.01$ ($\Omega=0$ and $T_f\approx 9600$).}}
  \label{fig6}
\end{subfigure}
\caption{Numerical approximation of the solutions to \eqref{new-Y}-\eqref{new-Y-IC}.}
\label{figures}
\end{figure}

Taking into account the graphs obtained for several values of the parameters of our system, we list here the main interesting observations that we could reveal from this numerical test. 

\begin{itemize}
\item[($i$)] For all the cases in Figure \ref{figures}, as it is obtained in Theorem \ref{blowup-equality} (see \eqref{assump}), the chosen values for the different parameters give rise to $\Omega(N,\mu_1,\mu_2,p,q) \ge 0$, where $\Omega(N,\mu_1,\mu_2,p,q)$ is defined by \eqref{blow-up-reg}. Moreover, in Figures \ref{fig5}-- \ref{fig4}, we observe that the quantities $Y_1(t)$ and $Y_2(t)$ blow up simultaneously at the same time. 
\item[($ii$)] In Figures \ref{fig1} and \ref{fig2}, all the parameters are set the same except the initial data size, $\varepsilon$, which is intended to decrease toward zero. Obviously, the blow-up time is inversely proportional to  $\varepsilon$. For $\varepsilon = 0.1$, the blow-up time $T_b \approx 37$, however, it is about $990$ for $\varepsilon = 0.01$.
\item[($iii$)] Compared to Case \ref{fig1} (resp. Case \ref{fig2}), Fig. \ref{fig3} (resp. Fig. \ref{fig4}) shows the effects of the absence of one of the damping terms (here $\mu_1=0$) on the blow-up time which is, hence, weakened. We recover here the fact that the damping is enhancing the blow-up time. 
\item[($iv$)] In Figures \ref{fig5} and \ref{fig6}, where the quantity  $\Omega$ is set to be $0$, the behavior of the final time is different compared to the other cases. This is indeed the almost-global existence case. 
\end{itemize}

\section*{\bf\Large Declarations and Statements}

\noindent{\bf\large Competing interests.}
The authors declare that they have no known competing financial interests or personal relationships that could have appeared to influence the work reported in this paper.
\medskip

\noindent{\bf\large Authors' contributions.}
Both authors wrote and revised the paper.
\medskip

\noindent{\bf\large Funding.}
Funding information is not applicable / No funding was received. 
\medskip

\noindent{\bf\large Availability of data and materials.} 
No data-sets were generated or analyzed during the current study.
\medskip

\hrule  

\bibliographystyle{plain}

\end{document}